\documentclass{amsart}
\usepackage{fontenc}
\usepackage{url}
\usepackage{hyperref}
\newtheorem*{conjecture}{Conjecture}
\newtheorem{theorem}{Theorem}
\newtheorem{corollary}[theorem]{Corollary}
\newtheorem{proposition}[theorem]{Proposition}
\newtheorem{lemma}[theorem]{Lemma}
\title{Entropy approach for a generalization of Frankl's conjecture}
\author{Veronica Phan*}
\thanks{*Ho Chi Minh City; email: \url{kyubivulpes@gmail.com}}
\begin{document}
\begin{abstract}
In this paper, we will use the entropy approach to derive a necessary and sufficient condition for the existence of an element that belongs to at least half of the sets in a finite family of sets.
\end{abstract}
\maketitle
\section{Introduction}
The union-closed set conjecture, or Frankl's conjecture, is a famous open problem in combinatorics. A family of set $\mathcal{F}$ is said to be union-closed if the union of any two sets from $\mathcal{F}$ belongs to $\mathcal{F}$ as well.
\begin{conjecture}
[Frankl]
For every finite union-closed family of sets $\mathcal{F}\neq\{\emptyset\}$, there exists an element $i$ such that at least half of the sets in $\mathcal{F}$ contain $i$.
\end{conjecture}
In 2022, Gilmer \cite{Gilmer} used entropy method to establish the first constant lower bound $0.01$ for this conjecture, soon later, three preprints \cite{CL}\cite{AHS}\cite{Sawin} improved the bound to $\frac{3-\sqrt{5}}{2}$, and Lei Yu \cite{Yu} and Cambie \cite{Cambie} improved it to about $0.38234$.

In this paper, we won't just study union-closed family, but arbitrary finite family of sets as well. For $S\in[n]$ and family $\mathcal{F}\subseteq2^{[n]}$, denote $\mathcal{F}(S)$ be the family of sets which have the form $S\cup F, F\in\mathcal{F}$. Our main result is:
\begin{theorem}
\label{main}
For family $\mathcal{F}\subseteq2^{[n]}$, there exists an element $i$ such that at least half of the sets in $\mathcal{F}$ contain $i$ if and only if there exists a family $\mathcal{G}\subseteq2^{[n]},|\mathcal{G}|>1$ such that $$\sum_{S\in\mathcal{F}}\log|\mathcal{G}(S)|\leq\frac{|\mathcal{F}|\log|\mathcal{G}|}{2}$$.
\end{theorem}
\section{Notation and Preliminaries}
All the $\log$ in this paper are of base $2$. For a random variable $X$ valued in sets and a set $S$, we could define random variables $X\cup S,X/S,...$ in a natural way, and denote $H(X),|X|$ be the entropy and the number of possible value of $|X|$ For a finite family $\mathcal{F}$ of set, let $X_{\mathcal{F}}$ be a random variable of sets sampled uniformly at random from $\mathcal{F}$.

We will use these following properties of entropy:
\begin{enumerate}
\item $0\leq H(X)\leq \log|X|$
\item $H(X,Y,Z)+H(Z)\leq H(X,Z)+H(Y,Z)$
\item $H(X|f(X))+H(f(X))=H(X,f(X))=H(X)$ for function $f$. 
\end{enumerate}
\section{Main lemma}
First, we will prove the following lemma:
\begin{lemma}
\label{mainlemma}
Let $X$ be a random variable of sets sampled from $2^{[n]}$, then there exists $n$ nonnegative numbers $x_1,x_2,...,x_n$ such that $H(X)=\sum_{i=1}^nx_i$ and for every set $S\subseteq[n]$, we have: $$H(X\cup S)\geq\sum_{i\in[n]/S}x_i$$.
\end{lemma}
\begin{proposition}
\label{basic}
For every set $S\subseteq[n], n\notin S$, we have:
$$H(X)-H(X\cup\{n\})\leq H(X\cup S)-H(X\cup S\cup\{n\})$$
\end{proposition}
\begin{proof}
For $R\subseteq[n]$, if we know $R\cup S$ and $R\cup\{n\}$, we could determine $R$, which is $R\cup\{n\}$ if $n\in R\cup S$ and $(R\cup\{n\})/\{n\}$ if otherwise. From this, we could deduce that $H(X)=H(X,X\cup S,X\cup\{n\})=H(X\cup S,X\cup\{n\})$.

Now we have:
\begin{equation*}
\begin{split}
& H(X\cup S,X\cup\{n\},X\cup S\cup\{n\})+H(X\cup S\cup\{n\})\leq H(X\cup S,X\cup S\cup\{n\})+H(X\cup\{n\},X\cup S\cup\{n\}) \\ & \Rightarrow H(X)+H(X\cup S\cup\{n\})\leq H(X\cup S)+H(X\cup\{n\}) \\ & \Rightarrow H(X)-H(X\cup\{n\})\leq H(X\cup S)-H(X\cup S\cup\{n\})
\end{split}
\end{equation*}
\end{proof}
\begin{proof}
[Proof of lemma \ref{mainlemma}]
We will prove by induction on $n$. For $n=1$, we could choose $x_1=H(X)$ and it's easy to show $x_1$ satisfies the lemma.

Assume the lemma is true with $n=k-1$. For $n=k$ and random variable $X$ of sets sampled from $2^{[n]}$, consider $(X\cup\{n\})/\{n\}$ and view it as random variable of sets sampled from $2^{[n-1]}$. Apply the induction hypothesis, there exists nonnegative numbers $x_1,x_2,...,x_{n-1}$ such that $H(X\cup\{n\})=H((X\cup\{n\})/\{n\})=\sum_{i=1}^{n-1}x_i$ and for every set $S\in[n-1]$, we have: $$H(X\cup S\cup\{n\})=H(((X\cup\{n\})/\{n\})\cup S)\geq\sum_{i\in[n-1]/S}x_i$$
Now we take $x_n=H(X|X\cup\{n\})=H(X)-H(X\cup\{n\})\geq0$, then we just need to check $H(X\cup S)\geq\sum_{[n]/S}x_i$ for every set $S\subseteq[n],n\notin S$. By proposition \ref{basic} we have:
\begin{equation*}
\begin{split}
\sum_{i\in[n]/S}x_i & =\sum_{i\in[n-1]/S}x_i+x_n \\ & \leq H(X\cup S\cup\{n\})+(H(X\cup S)-H(X\cup S\cup\{n\})) \\ & =H(X\cup S)
\end{split} 
\end{equation*}
as we want.
\end{proof}
\section{Proof of the main theorem}
For finite family of sets $\mathcal{F}$, denote $w_{\mathcal{F}}(i)$ be the number of sets in $\mathcal{F}$ that not contain $i$.
\begin{proof}
[Proof of theorem \ref{main}]
\textbf{Only if:} Take $i\in[n]$ such that at least half of the sets in $\mathcal{F}$ contain $i$, then $w_{\mathcal{F}}(i)\leq\frac{|\mathcal{F}|}{2}$. We take $\mathcal{G}=\{\emptyset,\{i\}\}$, then $|\mathcal{G}|=2$, for $S\in\mathcal{F}$, $|\mathcal{G}(S)|=|\{S\}|=1$ if $i\in S$ and $|\mathcal{G}(S)|=|\{S,S\cup\{i\}|=2$ if otherwise, so we have: 
\begin{equation*}
\begin{split} \sum_{S\in\mathcal{F}}\log|\mathcal{G}(S)| & = \sum_{S\in\mathcal{F},i\in S}\log|\mathcal{G}(S)|+\sum_{S\in\mathcal{F},i\notin S}\log|\mathcal{G}(S)| \\ & =(|\mathcal{F}|-w_{\mathcal{F}}(i))\log1+w_{\mathcal{F}}(i)\log2=w_{\mathcal{F}}(i) \\ & \leq \frac{|\mathcal{F}|}{2}=\frac{|\mathcal{F}|\log|\mathcal{G}|}{2}
\end{split}
\end{equation*}
\textbf{If:} For family of sets $\mathcal{G}$ satisfies the condition of the theorem, take $n$ nonnegative numbers $x_1,x_2,...,x_n$ by applying lemma \ref{mainlemma} for the random variable $X_{\mathcal{G}}$, then we have:
\begin{equation*}
\begin{split}
\sum_{i=1}^nx_iw_{\mathcal{F}}(i) & =\sum_{i=1}^n\sum_{S\in\mathcal{F}, i\notin S}x_i \\ & =\sum_{S\in\mathcal{F}}\sum_{i\in[n]/S}x_i \\ & \leq\sum_{S\in\mathcal{F}}H(X_{\mathcal{G}}\cup S) \\ & \leq\sum_{S\in\mathcal{F}}\log|X_{\mathcal{G}}\cup S| \\ & =\sum_{S\in\mathcal{F}}\log|\mathcal{G}(S)| \\ & \leq\frac{|\mathcal{F}|\log|\mathcal{G}|}{2} \\ & =\frac{|\mathcal{F}|H(X_{\mathcal{G}})}{2}=\frac{|\mathcal{F}|}{2}\sum_{i=1}^nx_i
\end{split}
\end{equation*}
so there exists $i\in[n]$ such that $w_{\mathcal{F}}(i)\leq\frac{|\mathcal{F}|}{2}$ (as $\sum_{i=1}^nx_i=H(X_{\mathcal{G}})=\log|\mathcal{G}|>0$), in other word, at least half of the sets in $\mathcal{F}$ contain $i$
\end{proof}
\begin{corollary}
The union-closed conjecture is true for finite union-closed family of sets $\mathcal{F}$ if there exists subfamily $\mathcal{G}\subseteq\mathcal{F},|\mathcal{G}|>1$ such that: $$\sum_{S\in\mathcal{F}}\log|\mathcal{G}(S)|\leq\frac{|\mathcal{F}|\log|\mathcal{G}|}{2}$$
\end{corollary}
This corollary does not depend on the base set but only the union structure of the family $\mathcal{F}$. We have some strategies to prove the union-closed conjecture for $\mathcal{F}$

A natural way to choose the subfamily $\mathcal{G}$ is the family $\mathcal{F}$ itself, but the condition $\sum_{S\in\mathcal{F}}\log|\mathcal{F}(S)|\leq\frac{|\mathcal{F}|\log|\mathcal{F}|}{2}$ is not always true, for example $\mathcal{F}=\{\emptyset,\{1\},\{1,2\}\}$. We may need another quantity on random variable valued on sets which behave like entropy to prove the union-closed conjecture this way.

Another strategy is note that the union-closed conjecture is true for $\mathcal{F}$ if it is true for $\mathcal{F}^N$ for some $N$. For very small $\epsilon>0$ and very large $N$ depend on $\epsilon$, take $\mathcal{G}$ be the subfamily of $\mathcal{F}^N$ consist all the set $S$ such that $|\mathcal{F}^N(S)|\geq|\mathcal{F}|^{\epsilon(N-2)}$. We expect that for many cases, $|\mathcal{G}|\leq(\frac{1}{2}-\epsilon)|\mathcal{F}^N|$, then:
\begin{equation*}
\begin{split}
\sum_{S\in\mathcal{F}^N}\log|\mathcal{G}(S)| & =\sum_{S\in\mathcal{G}}\log|\mathcal{G}(S)|+\sum_{S\in\mathcal{F}^N/\mathcal{G}}\log|\mathcal{G}(S)| \\ & \leq\sum_{S\in\mathcal{G}}\log|\mathcal{G}|+\sum_{S\in\mathcal{F}^N/\mathcal{G}}\log|\mathcal{F^N}(S)| \\ & \leq|\mathcal{G}|\log|\mathcal{G}|+|\mathcal{F}^N|\log|\mathcal{F}|^{\epsilon(N-2)} \\ & \leq(\frac{1}{2}-\epsilon)|\mathcal{F}^N|\log|\mathcal{G}|+|\mathcal{F}^N|(\log|\mathcal{F}|^{N\epsilon}-\log|\mathcal{F}|^{2\epsilon}) \\ & \leq(\frac{1}{2}-\epsilon)|\mathcal{F}^N|\log|\mathcal{G}|+|\mathcal{F}^N|(\log|\mathcal{F}^N|^{\epsilon}-\log2^{2\epsilon}) \\ & =(\frac{1}{2}-\epsilon)|\mathcal{F}^N|\log|\mathcal{G}|+\epsilon(|\mathcal{F}^N|(\log|\mathcal{F}^N|-2) \\ & \leq(\frac{1}{2}-\epsilon)|\mathcal{F}^N|\log|\mathcal{G}|+\epsilon(|\mathcal{F}^N|(\log|\mathcal{G}|)=\frac{|\mathcal{F}^N|\log|\mathcal{G}|}{2}
\end{split}
\end{equation*}
so $\mathcal{F}^N$ as well as $\mathcal{F}$ statisfies the union-closed conjecture.


\begin{thebibliography}{9}
\bibitem{Gilmer}
Justin Gilmer. "A constant lower bound for the union-closed sets conjecture". arXiv: \href{https://arxiv.org/abs/2211.09055}{2211.09055}.
\bibitem{CL}
Zachary Chase, Shachar Lovett. "Approximate union closed conjecture". arXiv: \href{https://arxiv.org/abs/2211.11689}{2211.11689}.
\bibitem{AHS}
Ryan Alweiss, Brice Huang, Mark Sellke. "Improved Lower Bound for Frankl's Union-Closed Sets Conjecture". arXiv: \href{https://arxiv.org/abs/2211.11731}{2211.11731}.
\bibitem{Sawin}
Will Sawin. "An improved lower bound for the union-closed set conjecture". arXiv: \href{https://arxiv.org/abs/2211.11504}{2211.11504}.
\bibitem{Yu}
Lei Yu. "Dimension-Free Bounds for the Union-Closed Sets Conjecture". arXiv: \href{https://arxiv.org/abs/2212.00658}{2212.00658}.
\bibitem{Cambie}
Stijn Cambie. "Better bounds for the union-closed sets conjecture using the entropy approach". arXiv: \href{https://arxiv.org/abs/2212.12500}{2212.12500}.
\end{thebibliography}
\end{document}